\documentclass{amsart}
\usepackage{graphicx}
\usepackage{amsmath}
\usepackage{amssymb}
\usepackage{amsthm}
\usepackage{esint}

\newtheorem{theorem}{Theorem}[section]
\newtheorem{lem}[theorem]{Lemma}
\theoremstyle{Corollary}
\newtheorem{cor}[theorem]{Corollary}
\newtheorem{prop}[theorem]{Proposition}
\newtheorem{definition}[theorem]{Definition}
\newtheorem{example}[theorem]{Example}
\let\th\theta
\def\sl{\Delta_{\theta}}

\def\ff{|\nabla_{\theta} \varphi_\lambda|^2}
\def\ii{\sqrt{-1}}
\let\a\alpha
\let\b\beta

\numberwithin{equation}{section}

\begin{document}

\title{CR Yamabe constant, CR Yamabe flow and its soliton}

\author{Pak Tung Ho}
\address{Department of Mathematics, Sogang University, Seoul
04107, Korea}

\email{ptho@sogang.ac.kr, paktungho@yahoo.com.hk}

\author{Kunbo Wang}
\address{College of Sciences, China Jiliang University, Hangzhou {\rm310018}, China.}
\email{wangkb@cjlu.edu.cn, 21235005@zju.edu.cn}


\date{28th July, 2019.}


\begin{abstract}
On a  compact strictly pseudoconvex CR manifold $(M,\th)$, we consider the CR Yamabe constant of its infinite conformal covering.
By using the  maximum principles, we then prove a uniqueness theorem for the CR Yamabe flow
on a complete noncompact CR manifold.
Finally we obtain some properties of the CR Yamabe soliton on complete noncompact CR manifolds.

\end{abstract}

\maketitle

\section{Introducton}

Let $(M,\theta)$ be a compact pseudoconvex manifold of real dimension $2n+1$ with
the contact form $\theta$.
Given $0<u\in C^\infty(M)$, we can define the energy by
\begin{equation*}
E_{(M,\theta)}(u)=\int_M\left((2+\frac{2}{n})|\nabla_\theta u|^2+R_\theta u^2\right)dV_\theta
\Bigg/\left(\int_Mu^{2+\frac{2}{n}}dV_\theta\right)^{\frac{n}{n+1}}.
\end{equation*}
The CR Yamabe constant of $(M,\theta)$ is then  defined as
\begin{equation}\label{CR_Yamabe}
Y(M,\theta)=\inf\left\{E_{(M,\theta)}(u)~|~0<u\in C^\infty(M)\right\}.
\end{equation}
It follows from the definition that the CR Yamabe constant
depends only on the conformal class of $\theta$.
The CR Yamabe problem is to find a contact form in the conformal class such that
its Tanaka-Webster scalar curvature is constant.
If we write $\tilde{\theta}=u^{\frac{2}{n}}\theta$ for some $0<u\in C^\infty(M)$, then
their Tanaka-Webster scalar curvatures are related by
\begin{equation}\label{CR_Yamabe_eqn}
-(2+\frac{2}{n})\Delta_\theta u+R_\theta u
=R_{\tilde{\theta}}u^{1+\frac{2}{n}}.
\end{equation}
Therefore, the CR Yamabe problem is to find a positive smooth function $u$ in $M$
satisfying (\ref{CR_Yamabe_eqn}) with $R_{\tilde{\theta}}$ being constant.
The following theorem was proved by Jerison and Lee: (See Theorem 3.4 in \cite{Jerison&Lee3})

\begin{theorem}[Jerison-Lee]\label{thm1}
Given a compact CR manifold $(M,\theta)$, we have\\
(i) $Y(M,\theta)\leq Y(S^{2n+1},\theta_{S^{2n+1}})$ where $Y(S^{2n+1},\theta_{S^{2n+1}})$
is the CR Yamabe constant of the CR sphere $(S^{2n+1},\theta_{S^{2n+1}})$; and\\
(ii) if $Y(M,\theta)< Y(S^{2n+1},\theta_{S^{2n+1}})$
then the infimum in \eqref{CR_Yamabe} is attained by a positive
$C^\infty$ solution to \eqref{CR_Yamabe_eqn}. Thus, the contact form $\tilde{\theta}=u^{\frac{2}{n}}\theta$ has constant
Tanaka-Webster scalar curvature.
\end{theorem}

In view of Theorem \ref{thm1}, it is important to estimate
the CR Yamabe constant $Y(M,\theta)$. In section \ref{section2},
we consider the CR Yamabe constant of
conformal covering of $(M,\theta)$.

The flow approach has been introduced to study the CR Yamabe problem.
The normalized CR Yamabe flow is defined as
the evolution of the contact form $\theta(t)$:
\begin{equation}\label{0.1}
\frac{\partial}{\partial t}\theta(t)=-(R_{\theta(t)}-\overline{R}_{\theta(t)})\theta(t),~~\theta(t)|_{t=0}=\theta,
\end{equation}
where $R_{\theta(t)}$ is the Tanaka-Webster scalar curvature of $\theta(t)$,
and $\overline{R}_{\theta(t)}$ is the average of $R_{\theta(t)}$.
The CR Yamabe flow has been studied by many authors.
See \cite{Chang&Chiu&Wu,Ho2,Ho3,Ho4,Ho5,Sheng&Wang} and the references therein.

However, there has not been much work related to
the CR Yamabe problem or CR Yamabe flow on noncomapct manifold.
In \cite{Ho&Kim},
an existence result of the CR Yamabe problem on noncompact manifold was proved. As an analogue to the CR Yamabe flow on a compact CR manifold, one can define the (unnormalized) CR Yamabe flow as follows:
\begin{equation}\label{CRYF}
\frac{\partial}{\partial t}\theta(t)=-R_{\th(t)}\theta(t),
\end{equation}
where $R_{\theta(t)}$ is the Tanaka-Webster scalar curvature of $\theta(t)$.
By using the  maximum principles, we prove in section \ref{section3} a uniqueness theorem for the CR Yamabe flow \eqref{CRYF}
on noncompact manifold. See Theorem \ref{2.6}.

The CR Yamabe soliton is a self-similar solution to the (unnormalized) CR Yamabe flow.
More precisely,  $(M,\theta)$ is a CR Yamabe soliton if
there exist an infinitesimal contact diffeomorphism $X$ and a constant $\mu$ such that
\begin{equation}\label{1.1}
\begin{split}
R_\theta+\frac{1}{2}\mathcal{L}_X\theta&=\mu,\\
\mathcal{L}_XJ&=0.
\end{split}
\end{equation}
Here $R_\theta$ is the Tanaka-Webster scalar curvature of $\theta$, and $\mathcal{L}_X$ is
the Lie derivative with respect to $X$.
It follows from (\cite{Cao}) that, equivalently,
 $(M,\theta)$ is a CR Yamabe soliton if
there exist a real-valued smooth function $f$ and a constant $\mu$ such that
\begin{equation}\label{1.2}
\begin{split}
R_\theta+\frac{1}{2}f_0&=\mu,\\
f_{\alpha\alpha}+\sqrt{-1}A_{\alpha\alpha} f&=0~~\mbox{ for all }\alpha=1,2,...,n.
\end{split}
\end{equation}
Here $A_{\alpha_\beta}$ is the torsion of $\theta$, and
$f$ is called potential function. Sometimes we also say that $(M,\theta,f)$ is a CR Yamabe soliton
by specifying the potential function $f$.
The CR Yamabe soliton is called shrinking if $\mu>0$, steady if $\mu=0$, and expanding
if $\mu<0$ respectively.

The following lemma was proved in \cite{Cao}: (see Lemma 3.2 in \cite{Cao})

\begin{prop}\label{prop1.1}
A three-dimensional CR Yamabe soliton satisfies
\begin{equation}\label{1.5}
4\Delta_\theta R_\theta+2R_\theta(R_\theta-\mu)-(R_\theta)_0 f-\langle\nabla_\theta R_\theta, J(\nabla_\theta f)\rangle_{\theta}=0.
\end{equation}
\end{prop}

By using Proposition \ref{prop1.1}, Cao, Chang and Chen proved in
\cite{Cao} that any compact three-dimensional CR Yamabe soliton must have constant
Tanaka-Webster scalar curvature. See also \cite{Ho} for another proof which is valid for all dimensions.
In view of these results, one would like to study the CR Yamabe soliton on noncompact manifold.
See \cite{Cao,Cao1} for results related to the noncompact CR Yamabe soliton.
In section \ref{section4}, we study the noncompact CR Yamabe solitons.
In particular, we are able to classify the CR Yamabe solitons on the Heisenberg group $\mathbb{H}^n$
equipped with the standard contact form. See Theorem \ref{main}.

This paper is organized as follows. In section \ref{section1}, we recall some basic concepts in CR geometry. In section \ref{section2}, we consider the CR Yamabe constant of conformal covering of $(M,\th)$. In section \ref{section3}, we prove a uniqueness theorem for the CR Yamabe flow \eqref{CRYF} on noncompact manifold. Finally, in section \ref{section4}, we study the CR Yamabe solitons on a complete noncompact CR manifold.

\section{Preliminaries and Notations}\label{section1}

Let $M$ be an orientable, real, $(2n+1)$-dimensional manifold. A CR structure on $M$ is given by a complex $n$-dimensional
subbundle $T_{1,0}$ of the complexified tangent bundle ${\mathbb C}TM$ of $M$, satisfying $T_{1,0}\cap T_{0,1}=\{0\}$, where
$T_{0,1}=\bar{T}_{1,0}$. We assume the CR structure is integrable, that is, $T_{1,0}$ satisfies the formal Frobenius condition
$[T_{1,0}, T_{0,1}]\subset T_{1,0}$. We set $G=Re(T_{1,0}\oplus T_{0,1})$, so that $G$ is a real $2n$-dimensional subbundle of $TM$.
Then $G$ carries a natural complex structure map: $J: G\rightarrow G$ given by $J(V+\bar{V})=\ii(V-\bar{V})$ for $V\in T_{1,0}$.

Let $E\subset T^{\ast}M$ denote the real line bundle $G^{\bot}$. Because we assume $M$ is orientable, and the complex structure $J$ induces an
orientation on $G$, $E$ has a global nonvanishing section. A choice of such a 1-form $\th$ is called a pseudohermitian structure on $M$. Associated with each such $\theta$ is the real symmetric bilinear form $L_\theta$ on $G$:
\begin{equation*}
L_\theta(V,W)=d\theta(V,JW),~~V,W\in G
\end{equation*}
called the $Levi-form$ of $\theta$. $L_\theta$ extends by complex linearity to $\mathbb{C}G$, and induces a Hermitian form on $T_{1,0}$, which we write
\begin{equation*}
L_\theta(V,\bar W)=-\ii d\theta(V,\bar W),~~V,W\in T_{1,0}
\end{equation*}
If $\theta$ is replaced by $\tilde\theta=f\theta$, $L_\theta$ changes conformally by $L_{\tilde\theta}=fL_\theta$. We will assume that $M$ is {\it{strictly pseudoconvex}}, that is, $L_\theta$ is positive definite for a suitable $\theta$. In this case, $\theta$ defines a contact structure on $M$, and we call $\theta$ a contact form. Then we define the volume form on $M$ as $dV_{\theta}=\theta\wedge d\theta^n$.

We can choose a unique $T$ called the characteristic direction such that $\theta(T)=1$, $d\theta(T, \cdot)=0$, and $TM=G\oplus \mathbb{R}T$. Then we can define a coframe $\{\theta, \theta^1, \theta^2, \cdots, \theta^n\}$ satisfying $\theta^\alpha(T)=0$, which is called admissible coframe. And its dual frame $\{T, Z_1, Z_2, \cdots, Z_n\}$ is called admissible frame. In this coframe, we have $d\theta=\ii h_{\alpha\bar\beta}\theta^\alpha\wedge\theta^{\bar\beta}$, $h_{\alpha\bar\beta}$ is a Hermitian matrix.

The sub-Laplacian operator $\Delta_{\th}$ is defined by
$$\int_M (\Delta_{\th} u)fdV_{\th}=-\int_M\langle du,df\rangle_{\th}dV_{\th},$$
for all smooth function $f$. Here $\langle , \rangle_{\th}$ is the inner product induced by $L_{\th}$. And we denote $|\nabla_{\th}u|^2=\langle du,du\rangle_{\th}$. Tanaka \cite{T} and Webster \cite{W} showed there is a natural connection in the bundle $T_{1,0}$ adapted to a pseudohermitian structure, which is called the Tanaka-Webster connection. To define this connection, we choose an admissible coframe $\{\th^{\a}\}$ and dual frame $\{Z_{\a}\}$ for $T_{1,0}$. Then there are uniquely determined 1-forms $\omega_{\a\bar{\b}}$, $\tau_{\a}$ on $M$, satisfying
\begin{eqnarray}
d\theta^\alpha &=& \omega^\alpha_\beta\wedge\theta^\beta+\theta\wedge\tau^\alpha,\\
dh_{\alpha\bar\beta} &=& h_{\alpha\bar\gamma}\omega_{\bar\beta}^{\bar\gamma}+\omega_{\alpha}^{\gamma}h_{\beta\bar\gamma},\\
\tau_\alpha\wedge\theta^\alpha &=& 0.
\end{eqnarray}
From the third equation, we can find $A_{\alpha\gamma}$, such that
$$\tau_\alpha=A_{\alpha\gamma}\theta^\gamma$$
and $A_{\alpha\gamma}=A_{\gamma\alpha}$. Here $A_{\alpha\gamma}$ is called the pseudohermitian torsion.
With this connection, the covariant differentiation is defined by
$$
D Z_\alpha=\omega_\alpha^\beta\otimes Z_\beta,~~~~DZ_{\bar\alpha}=\omega_{\bar\alpha}^{\bar\beta}\otimes Z_{\bar\beta},~~~~DT=0.
$$
$\{\omega^{\alpha}_\beta\}$ are called connection 1-forms.
For a smooth function $f$ on $M$, we write $f_\alpha=Z_\alpha f,~~f_{\bar\alpha}=Z_{\bar\alpha} f,~~f_0=Tf$, so that
$df=f_\alpha \theta_\alpha+f_{\bar\alpha} \theta_{\bar\alpha}+f_0 \theta$. The second covariant differential $D^2 f$ is the 2-tensor with components
\begin{equation*}
\begin{split}
f_{\alpha\beta} &=\overline{\bar f_{\bar\alpha\bar\beta}}=Z_\beta Z_\alpha f-\omega_\alpha^\gamma(Z_\beta) Z_\gamma f, ~~f_{\alpha\bar\beta} =\overline{\bar f_{\bar\alpha\beta}}=Z_{\bar\beta} Z_\alpha f-\omega_\alpha^\gamma(Z_{\bar\beta}) Z_\gamma f,\\
f_{0\alpha} &=\overline{\bar f_{0\bar\alpha}}=Z_\alpha Tf,~~f_{\alpha0}=\overline{\bar f_{\bar\alpha 0}}=TZ_\alpha f-\omega_\alpha^\gamma(T) Z_\gamma f,~~f_{00}=T^2 f.
\end{split}
\end{equation*}
$h_{\a\bar{\b}}$ and $h^{\a\bar{\b}}$ are used to lower and raise the indices. The connections forms also satisfy
$$
d\omega_\beta^\alpha-\omega_\beta^\gamma\wedge\omega_\gamma^\alpha=\frac{1}{2}R_{\beta~~\rho\sigma}^{~~\alpha}\theta^\rho\wedge\theta^{\sigma}+
\frac{1}{2}R_{\beta~~\bar\rho\bar\sigma}^{~~\alpha}\theta^{\bar\rho}\wedge\theta^{\bar\sigma}
+R_{\beta~~\rho\bar\sigma}^{~~\alpha}\theta^\rho\wedge\theta^{\bar\sigma}+R_{\beta~~\rho 0}^{~~\alpha}\theta^\rho\wedge\theta-R_{\beta~~\bar\sigma 0}^{~~\alpha}\theta^{\bar\sigma}\wedge\theta.
$$
We call $R_{\beta\bar\alpha\rho\bar\sigma}$ the Tanaka-Webster curvature. Contractions of the Tanaka-Webster curvature yield the Tanaka-Webster Ricci curvature $R_{\rho\bar\sigma}=R_{\alpha~~\rho\bar\sigma}^{~~\alpha}$, or $R_{\rho\bar\sigma}=h^{\alpha\bar\beta}R_{\alpha\bar\beta\rho\bar\sigma}$, and the Tanaka-Webster scalar curvature $R=h^{\rho\bar\sigma}R_{\rho\bar\sigma}$.

And we denote $d_{\th}(\cdot,\cdot)$ to be the Carnot-Carath\'{e}odory distance with respect to $\th$.

\section{The CR Yamabe constant of conformal covering}\label{section2}

When the CR Yamabe constant $Y(M,\theta)<0$,
any two contact forms with constant Tanaka-Webster scalar curvature are identical
up to a scaling (see Theorem 7.1 in \cite{Jerison&Lee3} or Theorem 1.3 in \cite{Ho2} for example).
Hence, for a finite $k$-fold conformal covering $(\tilde{M},\tilde{\theta})$ of $(M,\theta)$,
we have
\begin{equation*}
Y(\tilde{M},\tilde{\theta})=k^{\frac{1}{n+1}}Y(M,\theta).
\end{equation*}

On the other hand, when the CR Yamabe constant $Y(M,\theta)>0$,
we cannot expect any similar explicit relations between
$Y(\tilde{M},\tilde{\theta})$ and $Y(M,\theta)$. One reason is that
the uniqueness for the constant Tanaka-Webster scalar curvature
mentioned above does not hold in general. However, we still have the following Aubin's Lemma for the CR case:

\begin{lem}[CR Aubin's Lemma]\label{Alem2}
Let $(M,\theta)$ be a compact pseudoconvex CR manifold of dimension $2n+1$
with $Y(M,\theta)>0$,
and $(\tilde{M},\tilde{\theta})$ a nontrivial finite conformal covering of $(M,\theta)$.
Then
$$Y(M,\theta)<Y(\tilde{M},\tilde{\theta}).$$
\end{lem}
\begin{proof}
For the contact form $\theta$, we consider its lift $\tilde{\theta}$.
Let $0<u\in C^\infty(\tilde{M})$ be a CR Yamabe minimizer with respect to $\tilde{\theta}$, i.e.
\begin{equation}\label{A1.1}
-(2+\frac{2}{n})\Delta_{\tilde{\theta}}u+R_{\tilde{\theta}} u=Y(\tilde{M},\tilde{\theta}) u^{1+\frac{2}{n}}~~\mbox{ on }\tilde{M}.
\end{equation}
Without loss of generality, we can assume
\begin{equation}\label{A1.0}
\int_{\tilde{M}}u^{2+\frac{2}{n}}dV_{\tilde{\theta}}=1.
\end{equation}
Let $\mathcal{G}$ be the deck transformation group for the normal covering $\tilde{M}\to M$.
Consider its average
\begin{equation}\label{A1.2}
v:=\sum_{\gamma\in\mathcal{G}}u\circ\gamma~~\mbox{ on }\tilde{M},
\end{equation}
and define $0<v_0\in C^\infty(M)$ to be the function whose lift  to $\tilde{M}$ is $v$.
Then
\begin{equation}\label{A1.3}
Y(M,\theta)\leq E_{(M,\theta)}(v_0)
=\frac{\int_M\left((2+\frac{2}{n})|\nabla_\theta v_0|^2+R_\theta v_0^2\right)dV_\theta}{\left(\int_Mv_0^{2+\frac{2}{n}}dV_\theta\right)^{\frac{n}{n+1}}}
\end{equation}
Since $(\tilde{M},\tilde{\theta})$ is a finite $k$-fold conformal covering of $(M,\theta)$,
we have
\begin{equation}\label{A1.4}
\int_Mv_0^{2+\frac{2}{n}}dV_\theta=\frac{1}{k}\int_{\tilde{M}}v^{2+\frac{2}{n}}dV_{\tilde{\theta}}
\end{equation}
and
\begin{equation}\label{A1.5}
\begin{split}
&\int_M\left((2+\frac{2}{n})|\nabla_\theta v_0|^2+R_\theta v_0^2\right)dV_\theta\\
&=\frac{1}{k}\int_{\tilde{M}}\left((2+\frac{2}{n})|\nabla_{\tilde{\theta}} v|^2+R_{\tilde{\theta}} v^2\right)dV_{\tilde{\theta}}\\
&=\frac{1}{k}\int_{\tilde{M}}\left(-(2+\frac{2}{n})\Delta_{\tilde{\theta}} v+R_{\tilde{\theta}} v\right)v\,dV_{\tilde{\theta}}\\
&=\frac{1}{k}\sum_{\gamma\in\mathcal{G}}\int_{\tilde{M}}\left(-(2+\frac{2}{n})\Delta_{\tilde{\theta}} (u\circ\gamma)+R_{\tilde{\theta}} (u\circ\gamma)\right)v\,dV_{\tilde{\theta}}\\
&=\frac{1}{k}Y(\tilde{M},\tilde{\theta})\int_{\tilde{M}}\sum_{\gamma\in\mathcal{G}}(u\circ\gamma)^{1+\frac{2}{n}}v\,dV_{\tilde{\theta}}
\end{split}
\end{equation}
where we have used (\ref{A1.1}) and (\ref{A1.2}). Substituting (\ref{A1.4}) and (\ref{A1.5})
into (\ref{A1.3}), we get
\begin{equation}\label{A1.6}
Y(M,\theta)\leq k^{-\frac{1}{n+1}}Y(\tilde{M},\tilde{\theta})
\frac{\int_{\tilde{M}}\sum_{\gamma\in\mathcal{G}}(u\circ\gamma)^{1+\frac{2}{n}}v\,dV_{\tilde{\theta}}}
{\left(\int_{\tilde{M}}v^{2+\frac{2}{n}}dV_{\tilde{\theta}}\right)^{\frac{n}{n+1}}}.
\end{equation}
By H\"{o}lder's inequality and
the inequality
$\left(\sum_{i=1}^ka_i^p\right)^{\frac{1}{p}}<\sum_{i=1}^k a_i$
for $k\geq 2$, $p>1$, $a_i>0$, we can compute
\begin{equation}\label{A1.7}
\begin{split}
\int_{\tilde{M}}\sum_{\gamma\in\mathcal{G}}(u\circ\gamma)^{1+\frac{2}{n}}v\,dV_{\tilde{\theta}}
&\leq\int_{\tilde{M}}\Big(\sum_{\gamma\in\mathcal{G}}(u\circ\gamma)^{2+\frac{2}{n}}\Big)^{\frac{1}{n+1}}
\Big(\sum_{\gamma\in\mathcal{G}}(u\circ\gamma)^{\frac{n+1}{n}}\Big)^{\frac{n}{n+1}}v\,dV_{\tilde{\theta}}\\
&<\int_{\tilde{M}}\Big(\sum_{\gamma\in\mathcal{G}}(u\circ\gamma)^{2+\frac{2}{n}}\Big)^{\frac{1}{n+1}}
\Big(\sum_{\gamma\in\mathcal{G}}u\circ\gamma\Big)v\,dV_{\tilde{\theta}}\\
&=\int_{\tilde{M}}\Big(\sum_{\gamma\in\mathcal{G}}(u\circ\gamma)^{2+\frac{2}{n}}\Big)^{\frac{1}{n+1}}
v^2dV_{\tilde{\theta}}\\
&\leq \left(\int_{\tilde{M}}v^{2+\frac{2}{n}}dV_{\tilde{\theta}}\right)^{\frac{n}{n+1}}
\Bigg(\int_{\tilde{M}}\sum_{\gamma\in\mathcal{G}}(u\circ\gamma)^{2+\frac{2}{n}}dV_{\tilde{\theta}}\Bigg)^{\frac{1}{n+1}}.
\end{split}
\end{equation}
It follows from  (\ref{A1.0}) that
\begin{equation}\label{A1.8}
\int_{\tilde{M}}\sum_{\gamma\in\mathcal{G}}(u\circ\gamma)^{2+\frac{2}{n}}dV_{\tilde{\theta}}
=k\int_{\tilde{M}}u^{2+\frac{2}{n}}dV_{\tilde{\theta}}=k.
\end{equation}
Now the assertion follows from combining (\ref{A1.6})-(\ref{A1.8}).
\end{proof}

For a noncompact CR manifold $(X,\theta)$,
the CR Yamabe constant of $(X,\theta)$ is defined by
\begin{equation*}
Y(X,\theta)=\inf\left\{E_{(X,\theta)}(u)~|~0<u\in C^\infty_c(X)\right\}.
\end{equation*}
Here, $C^\infty_c(X)$ is the space of all compactly supported smooth functions in $X$.

The argument of Jerison and Lee in Theorem \ref{thm1}
is still valid for any noncompact manifold:
\begin{equation}\label{A1.12}
Y(X,\theta)\leq Y(S^{2n+1},\theta_{S^{2n+1}}).
\end{equation}

We have the following definition:

Let $G$ be an infinite group and $H$ a subgroup of $G$ with infinite index.
Let $\{G_i\}_{i\geq 1}$ be an infinite sequence of subgroups of $G$.
We call $\{G_i\}_{i\geq 1}$ \textit{a descending chain of finite index subgroups tending to} $H$ if\\
(i) each $G_i$ is a finite index subgroup of $G$ with $G_i\supset H$;\\
(ii) $G=G_1\supsetneq G_2\supsetneq \cdots \supsetneq G_i\supsetneq G_{i+1}\supsetneq\cdots$; and\\
(iii) $\bigcap_{i=1}^\infty G_i=H$.

The following theorem corresponds to an analogue
of Lemma \ref{Alem2} for the CR Yamabe constants
of \textit{infinite} conformal coverings.
We identify each $\pi_1(M_k)$ and $\pi_1(M_\infty)$
with their projections to $\pi_1(M)$ in the following theorem.

\begin{theorem}\label{Amain}
Let $(M,\theta)$ be
a compact pseudoconvex CR manifold of dimension $2n+1$
with $Y(M,\theta)>0$. Let $(M_\infty,\theta_\infty)\to (M,\theta)$
be an infinite conformal covering such that
$\pi_1(M)$ has a descending chain of finite index subgroups tending to $\pi_1(M_\infty)$.
Then
\begin{equation}\label{A1.13}
Y(M,\theta)<Y(M_\infty,\theta_\infty).
\end{equation}
\end{theorem}
\begin{proof}
Let $0<u\in C^\infty(M)$ be a CR Yamabe minimizer
with respect to $\theta$, i.e.
$E_{(M,\theta)}(u)=Y(M,\theta)$.
By replacing $\theta$ by
$$\tilde{\theta}=\frac{u^{\frac{2}{n}}\theta}{\left(\int_Mu^{2+\frac{2}{n}}dV_{\theta}\right)^{\frac{1}{n+1}}},$$
we may assume that
$$R_\theta=Y(M,\theta)~~\mbox{ and }~~\int_MdV_\theta=1.$$
Consider the lift $\theta_\infty$ of $\theta$ to $M_\infty$.
Note that
\begin{equation}\label{A1.14}
R_{\theta_\infty}=R_\theta=Y(M,\theta)>0.
\end{equation}
The Folland-Stein embedding
$$S_1^2(M_\infty,\theta_\infty)\hookrightarrow L^{2+\frac{2}{n}}(M_\infty,\theta_\infty)$$
combined with (\ref{A1.14}) implies that $Y(M_\infty,\theta_\infty)>0$.
To see this,
it follows from the Folland-Stein embedding that there exists a uniform constant $C_0>0$ depending only on $(M_\infty,\theta_\infty)$
such that
\begin{equation}\label{A1.15}
\left(\int_{M_\infty}\varphi^{2+\frac{2}{n}}dV_{\theta_\infty}\right)^{\frac{n}{n+1}}
\leq C_0\int_{M_\infty}\left(|\nabla_{\theta_\infty} \varphi|^2+ \varphi^2\right)dV_{\theta_\infty}.
\end{equation}
for any $\varphi\in C^\infty_c(M_\infty)$.
Thus, for any $\varphi\in C^\infty_c(M_\infty)$, we have
\begin{equation*}
\begin{split}
E_{(M_\infty,\theta_\infty)}(\varphi)
&=\frac{\int_{M_\infty}\left((2+\frac{2}{n})|\nabla_{\theta_\infty} \varphi|^2+ R_{\theta_\infty}\varphi^2\right)dV_{\theta_\infty}}
{\left(\int_{M_\infty}\varphi^{2+\frac{2}{n}}dV_{\theta_\infty}\right)^{\frac{n}{n+1}}}\\
&\geq \min\Big\{2+\frac{2}{n},Y(M,\theta)\Big\}\frac{\int_{M_\infty}\left(|\nabla_{\theta_\infty} \varphi|^2+\varphi^2\right)dV_{\theta_\infty}}
{\left(\int_{M_\infty}\varphi^{2+\frac{2}{n}}dV_{\theta_\infty}\right)^{\frac{n}{n+1}}}\\
&\geq C_0\min\Big\{2+\frac{2}{n},Y(M,\theta)\Big\}
\end{split}
\end{equation*}
by (\ref{A1.14}) and (\ref{A1.15}).
This implies that $Y(M_\infty,\theta_\infty)\geq C_0\displaystyle\min\Big\{2+\frac{2}{n},Y(M,\theta)\Big\}>0$,
as we claim.

For the covering $M_\infty\to M$, there exist base points $p_\infty\in M_\infty$, $p_1\in M$ and
the projection map $\mathcal{P}:(M_\infty,p_\infty)\to (M,p_1)$ such that
$\pi_1(M_\infty)=\pi_1(M_\infty,p_\infty)$ and $\pi_1(M)=\pi_1(M,p_1)$.
Since $\pi_1(M)$ has a descending chain of finite index subgroups tending to $\pi_1(M_\infty)$,
for each $k\geq 2$, there exists a finite covering
$P_k:(M_k,p_k)\to (M,p_1)$ satisfying
$\bigcap_{k=1}^\infty\pi_1(M_k)=\pi_1(M_\infty)$ ($M_1=M$ and $\pi_1(M_k)=\pi_1(M_k,p_k)$)
and the following:\\
(i) $M_\infty$ is an infinite covering $\mathcal{P}_k:(M_\infty,p_\infty)\to (M_k,p_k)$ of each $M_k$; \\
(ii) $M_{k+1}$ is a non-trivial finite covering $P_{k,k+1}:(M_{k+1},p_{k+1})\to (M_k,p_k)$ of each $M_k$ for $k\geq 1$.\\
Here we identify $\pi_1(M_\infty)$ and $\pi_1(M_k)$ with their projections to $\pi_1(M,p_1)$.
Denote the lifting of $\theta$ to $M_k$ by $\theta_k$. From Theorem \ref{thm1}(i) and Lemma \ref{Alem2}, we have
\begin{equation}\label{A1.17}
Y(M,\theta)<Y(M_2,\theta_2)<\cdots<Y(M_k,\theta_k)<Y(M_{k+1},\theta_{k+1})<\cdots\leq  Y(S^{2n+1},\theta_{S^{2n+1}}),
\end{equation}
and hence the limit of $\{Y(M_k,[\theta_k])\}_{k\geq1}$ always exists.
Therefore, by taking a suitable subsequence $\{Y(M_{k_j},[\theta_{k_j}])\}_{j\geq1}$
if necessary, it suffices to consider only the subsequence for the proof.
Throughout this paper, we always assume that
$(M_{k_1},\theta_{k_1})=(M,\theta)$ for such a subsequence $\{(M_{k_j},\theta_{k_j})\}_{j\geq1}$.

For each $k\geq 1$, we set
\begin{equation*}
\mathcal{D}_k=\Big\{x\in M_\infty \Big| d_{\theta_\infty}(x,p_\infty)<d_{\theta_\infty}(x,q_\infty)
~~\mbox{ for all }q_\infty\in\mathcal{P}_k^{-1}(p_k)-\{p_\infty\}\Big\}.
\end{equation*}
Here $d_{\theta_\infty}(\cdot,\cdot)$ is the Carnot-Carath\'{e}odory distance function with respect to $\theta_\infty$.
Since $M_k$ is compact and $\theta_\infty$ is the lifting of $\theta$  on $M=M_1$
to $M_\infty$, the infimum
$$\inf\Big\{d_{\theta_\infty}(q_\infty,\tilde{q}_\infty)\Big|q_\infty\neq\tilde{q}_\infty, ~~ q_\infty,\tilde{q}_\infty\in\mathcal{P}_k^{-1}(p_k)\Big\}$$
is strictly positive. Hence, there exists a finite subset $\mathcal{A}_k\subset \mathcal{P}_k^{-1}(p_k)-\{p_\infty\}$
for each $k\geq 1$ such that
\begin{equation*}
\mathcal{D}_k=\Big\{x\in M_\infty \Big| d_{\theta_\infty}(x,p_\infty)<d_{\theta_\infty}(x,q_\infty)
~~\mbox{ for all }q_\infty\in\mathcal{A}_k\Big\}.
\end{equation*}
Therefore, $\mathcal{D}_k$ is an open set, and especially a fundamental domain containing $p_\infty$
for the covering $M_\infty\to M_k$ such that
\begin{equation}\label{A1.16}
\mathcal{D}_1\subset \mathcal{D}_2\subset\cdots\subset
\mathcal{D}_k\subset \mathcal{D}_{k+1}\subset \cdots M_\infty,~~
\bigcup_{k\geq 1}\mathcal{D}_k=M_\infty.
\end{equation}
From (\ref{A1.16}), by taking a suitable subsequence if necessary, we can assume the following:\\
(iii) For all $k\geq 1$, the closure $\widehat{M}_k:=\overline{\mathcal{D}}_k \subset M_\infty$
satisfies
$d_{\theta_\infty}(\widehat{M}_k,\partial\widehat{M}_{k+1})\geq 1.$

By (iii), we can take a sequence $\{\Omega_k\}_{k\geq 2}$ of domains in $M_\infty$
with smooth boundary $\partial\overline{\Omega}_k$ satisfying
$$\widehat{M}_{k-1}\subset\Omega_k\subset\widehat{M}_{k}~~\mbox{ for }k\geq 2.$$
We will denote $\widehat{M}:=\widehat{M}_{1}$.

When $Y(M_\infty,\theta_\infty)= Y(S^{2n+1},\theta_{S^{2n+1}})$, it follows from (\ref{A1.17}),
Theorem \ref{thm1}(i) and Lemma \ref{Alem2} that
\begin{equation*}
Y(M,\theta)<Y(M_2,\theta_2)\leq  Y(S^{2n+1},\theta_{S^{2n+1}})=Y(M_\infty,\theta_\infty),
\end{equation*}
which proves (\ref{A1.13}). Hence, we may assume
$Y(M_\infty,\theta_\infty)<Y(S^{2n+1},\theta_{S^{2n+1}})$ by (\ref{A1.12}).

By definition of $Y(M_\infty,\theta_\infty)$ and the condition (iii), there exists $k_0\in\mathbb{N}$ such that
$$Q_k:=\inf_{\varphi\in C_c^\infty(\Omega_k)}E_{(M_\infty,\theta_\infty)}(\varphi)<Y(S^{2n+1},\theta_{S^{2n+1}})
~~\mbox{ for all }k\leq k_0$$
and that
$$Q_{k_0}>Q_{k_0+1}>\cdots>Q_{k_0+i}>Q_{k_0+i+1}>\cdots,
~~\lim_{k\to\infty} Q_k=Y(M_\infty,\theta_\infty).$$
By combining the inequality $Q_k<Y(S^{2n+1},\theta_{S^{2n+1}})$
with the same argument of Jerison and Lee in proving
Theorem \ref{thm1}(ii) for the compact manifolds,
the CR Yamabe problem of Dirichlet-type can be solved for each
$(\overline{\Omega}_k,\theta_\infty|_{\overline{\Omega}_k})$. Namely, there
exits $\psi_k\in C^\infty(\overline{\Omega}_k)$ such that
\begin{equation*}
E_{(M_\infty,\theta_\infty)}(\psi_k)=Q_k,~~
\psi_k>0\mbox{ in }\Omega_k~~\mbox{ and }~~\psi_k=0\mbox{ on }\partial\overline{\Omega}_k.
\end{equation*}

We denote the zero extension of $\psi_k$ to $M_\infty$ also by
$\psi_k\in C^{0,1}(M_\infty)\cap S_1^2(M_\infty,\theta_\infty)$.
Because $\psi_k|_{\partial\overline{\Omega}_k}=0$,
$\psi_k$ can also be regarded as a function on the compact manifold $M_k$.
Then
$$Y(M_k,\theta_k)\leq E_{(M_k,\theta_k)}(\psi_k)
=E_{(M_\infty,\theta_\infty)}(\psi_k)=Q_k,$$
and hence
$$\limsup_{k\to\infty}Y(M_k,\theta_k)\leq\lim_{k\to\infty}Q_k=Y(M_\infty,\theta_\infty).$$
From Lemma \ref{Alem2}, we obtain
$$Y(M,\theta)<Y(M_2,\theta_2)\leq
\lim_{k\to\infty}Y(M_k,\theta_k)=\limsup_{k\to\infty}Y(M_k,\theta_k)
\leq Y(M_\infty,\theta_\infty).$$
This completes the proof of (\ref{A1.13}).
\end{proof}

\section{The CR Yamabe flow on complete noncompact manifolds}\label{section3}

Let $(M,\theta)$ be a smooth, strictly pseudoconvex $(2n+1)$-dimensional complete noncompact CR manifold.
If we write $\th(t)=u^{\frac{2}{n}}\th$
for some $u=u(t)$, then it follows from (\ref{CR_Yamabe_eqn}) that
the CR Yamabe flow (\ref{CRYF}) reduces to
\begin{equation}\label{CRYF1}
\frac{\partial u}{\partial t}=(n+1)u^{-\frac{2}{n}}\sl u-\frac{n}{2}R_{\th} u^{1-\frac{2}{n}},
\end{equation}
from which we obtain
\begin{eqnarray}\label{CRYF2}
\frac{\partial}{\partial t}\left(u^{\frac{n+2}{n}}\right)=\frac{(n+1)(n+2)}{n}\left(\sl u- \frac{n}{2n+2}R_{\th} u\right).
\end{eqnarray}
Now we will establish two maximum principles which hold on complete noncompact CR manifolds. By using the  maximum principles, we prove a uniqueness theorem for the CR Yamabe flow (\ref{CRYF}). To do this, we follow the idea of Ma and An in \cite{maan}.
First we give two definitions.
\begin{definition}\label{A}
We say that a complete noncompact CR manifold $(M,\theta)$ satisfies \textit{Condition A}, if there exist a constant $A$ and a smooth positive function $f$ such that $f\rightarrow \infty$ near infinity and
\begin{equation*}
|\nabla_\theta f|\leq A \hspace{2mm}\mbox{ and} \hspace{2mm} \sl f\leq A~~\mbox{ on }M.
\end{equation*}
\end{definition}

\begin{definition}\label{B}
Let $t\in [0,T]$, we say $(M,\theta(t))$ satisfy \textit{Condition B}, if there exist a positive constant $B$ and a smooth positive function $g$ such that
\begin{equation*}
-\frac{\partial}{\partial t}g+\sl g\leq B,  \hspace{2mm}\mbox{on} \hspace{2mm} M\times [0,T],
\end{equation*}
and for any constant $C>0$, $x_0\in M$, there is a positive constant $d_0$ satisfying $g(x,t)>C$ for $d_{\theta(t)}(x_0,x)>d_0$ for all $t\in [0,T]$.
\end{definition}

\begin{example}
We claim that the Heisenberg group $\mathbb H ^n$ satisfies \textit{Condition A} and \textit{Condition B} above.
Recall that the Heisenberg group $\mathbb H ^n$ is a Lie group whose underlying manifold is $\mathbb C^n\times \mathbb R$ with coordinates $(z,t)=(z^1,z^2,\cdots,z^n,t)$. The contact form of $\mathbb H ^n$ is
$$\th_0=dt+\ii\sum_{\alpha=1}^n (z^\alpha d\bar{z}^\alpha-\bar{z}^\alpha dz^\alpha).$$
Then the dual frame $\{Z_\alpha\}$ for $T^{1,0}$ defined in section \ref{section1}
are given by
$$Z_\alpha=\frac{1}{\sqrt{2}}\left(\frac{\partial}{\partial z^\alpha}+\sqrt{-1}\bar{z}^\alpha\frac{\partial}{\partial t}\right)~~\mbox{ for }1\leq \alpha\leq n.$$
And $Z_{\bar{\alpha}}=\overline{Z}_{\alpha}$.
If we define
$f(z,t)=(|z|^4+t^2)^{p}$
for $(z,t)\in\mathbb H ^n$, where $0<p\leq 1/4$, then
$f \rightarrow \infty$ near infinity, and
\begin{equation}
\begin{split}
Z_\alpha f&=\frac{p(|z|^2+\sqrt{-1} t)\bar{z}_\alpha}{\sqrt{2}(|z|^4+t^2)^{1-p}},\\
Z_{\bar{\alpha}} Z_\alpha f&=\frac{p(|z|^2+\sqrt{-1} t)}{2(|z|^4+t^2)^{1-p}}+
\frac{p|z_\alpha|^2}{(|z|^4+t^2)^{1-p}}
-\frac{p(1-p)|z_\alpha|^2}{2(|z|^4+t^2)^{1-p}}
\end{split}
\end{equation}
for all $1\leq \alpha\leq n$. This implies that
$$|\nabla_{\theta_0} f|=\sum_{\alpha=1}^nZ_\alpha fZ_{\bar{\alpha}} f\leq A~~\mbox{ and }~~
\Delta_{\theta_0}f=\sum_{\alpha=1}^n(Z_{\bar{\alpha}} Z_\alpha+Z_\alpha Z_{\bar{\alpha}})f
\leq A
$$
for some uniform constant $A$, since $0<p\leq 1/4$.
Thus the Heisenberg group $\mathbb H ^n$ is an example satisfies \textit{Condition A}, and also satisfies \textit{Condition B}  by fixing the contact form $\theta(t)=\theta$
and the function $g=f$.
\end{example}

Now we prove the following weak maximum principle:
\begin{theorem}
Let $(M,\theta)$ be a complete noncompact CR manifold satisfying \textit{Condition A}. Assume $u\in C^2(M\times [0,T])$ satisfies the following evolution equation:
\begin{equation*}
\begin{split}
\frac{\partial}{\partial t}u&=\sl u +F(Du,u,x,t),\hspace{2mm}\mbox{ on} \hspace{2mm} M\times [0,T],\\
|u|&\leq C,\hspace{2mm}\mbox{ on} \hspace{2mm} M\times [0,T],\\
u(x,0)&\leq 0,\hspace{2mm}\mbox{ on} \hspace{2mm} M,\\
F(Du,u,x,t)&\leq C u(x,t),\hspace{2mm}\mbox{for} \hspace{2mm} u(x,t)>0.
\end{split}
\end{equation*}
Here $C>0$ is a fixed constant. Then we have $u\leq 0$ on $M\times [0,T]$.
\end{theorem}
\begin{proof}
We argue by contradiction. Assume that there is a point $(x_0,t_0)\in M\times [0,T]$ such that
$$u(x_0,t_0)>0.$$
We denote $F(x,t)=F(Du,u,x,t)$. Let $\psi_\lambda (x,t)=e^{\lambda t}(f(x)+At+1)$, where
$\lambda$ is a constant chosen
to satisfy $\lambda> C$. Then we have
\begin{equation*}
\begin{split}
\frac{\partial \psi_\lambda}{\partial t}&=\lambda \psi_\lambda +e^{\lambda t}A\\
&\geq \lambda\psi_\lambda+\sl \psi_\lambda.
\end{split}
\end{equation*}
Let $\varphi_\lambda =\frac{1}{\psi_\lambda}$, i.e. $0\leq \varphi_\lambda \leq 1$.
From which we get
\begin{equation*}
\begin{split}
\frac{\partial \varphi_\lambda}{\partial t}=-\psi_\lambda^{-2}\frac{\partial \psi_\lambda}{\partial t}&\leq -\psi_\lambda^{-2}(\lambda \psi_\lambda +\sl \psi_\lambda)\\
&=-\lambda \varphi_\lambda-\psi_\lambda^{-2}\sl \psi_\lambda.
\end{split}
\end{equation*}
And
\begin{equation*}
\begin{split}
\sl \varphi_\lambda&=-\psi_\lambda^{-2}\sl \psi_\lambda +2\psi_\lambda^{-3}|\nabla \psi_\lambda|^2\\
&=-\psi_\lambda^{-2}\sl \psi_\lambda+2\varphi_\lambda^{-1}\ff.
\end{split}
\end{equation*}
From the above, we can deduce that
\begin{equation*}
-\psi_\lambda^{-2}\sl \psi_\lambda=\sl \varphi_\lambda-2\varphi_\lambda^{-1}\ff.
\end{equation*}
Then we have
\begin{equation}\label{2.1}
\frac{\partial \varphi_\lambda}{\partial t}\leq \sl \varphi_\lambda-2\varphi_\lambda^{-1}\ff-\lambda \varphi_\lambda\hspace{2mm} \mbox{on} \hspace{2mm} M.
\end{equation}
Set $v=\varphi_\lambda u$. Then we have $v(x_0,t_0)=\varphi_\lambda(x_0,t_0)u(x_0,t_0)>0$. Since $|u|\leq C$ and $0<\varphi_\lambda\leq 1$, we can assume
$$a=\sup_{M\times [0,T]} v>0.$$
We denote
$$D=\left\{x\in M: d_{\th}(x_0,x)\leq \frac{C_1}{a}\right\}.$$
Here $C_1$ is a positive constant. Since $\varphi_\lambda$ tends to $0$ near infinity, we can choose $C_1$ large enough such that
\begin{equation*}
v(x,t)<a,\hspace{2mm} \mbox{for} \hspace{2mm} (x,t)\notin D\times [0,T].
\end{equation*}
Since $D\times [0,t]$ is compact, we can find an interior point $(x_1,t_1)\in D\times [0,T]$ such that
$$\sup_{M\times [0,T]}v(x,t)=v(x_1,t_1)=a.$$
At the point $(x_1,t_1)$, we have
$$\frac{\partial v}{\partial t}-\sl v\geq 0,$$
and
$$\nabla_{\theta}v=0.$$
Then by \eqref{2.1}, we have
\begin{equation*}
\begin{split}
\frac{\partial v}{\partial t}-\sl v&=-2\varphi_\lambda^{-1}\langle\ \nabla_{\theta} \varphi_\lambda, \nabla_{\theta} v\rangle+\varphi_\lambda F(x,t)+
(\frac{\partial \varphi_\lambda}{\partial t}-\sl \varphi_\lambda +\frac{2}{\varphi_\lambda}\ff)u\\
&\leq -2 \varphi_\lambda^{-1}\langle\ \nabla_{\theta} \varphi_\lambda, \nabla_{\theta} v\rangle+\varphi_\lambda F(x,t)-\lambda\varphi_\lambda u.
\end{split}
\end{equation*}
Now at the point $(x_1,t_1)$, we obtain
$$0\leq \varphi_\lambda F-\lambda \varphi_\lambda u.$$
Thus
$$F(x_1,t_1)\geq \lambda u(x_1,t_1)>Cu(x_1,t_1),$$
which is absurd by our assumption.
\end{proof}
Now we follow the idea of Ma and An \cite{maan} to prove the next maximum principle.
\begin{theorem}\label{2.7}
Let $(M,\theta)$ be a complete noncompact CR manifold. And $(M,\theta(t))$ satisfies \textit{Condition B} on $M\times [0,T]$. Assume $u\in C^2(M\times [0,T])$ satisfies the following evolution equation:
\begin{equation*}
\begin{split}
\frac{\partial}{\partial t}u-\sl u&\leq 0,\hspace{2mm}\mbox{ on} \hspace{2mm} M\times [0,T],\\
u(0)=u_0&\leq C,\hspace{2mm}\mbox{ on} \hspace{2mm} M.
\end{split}
\end{equation*}
Then $u\leq C$ on $M\times [0,T]$.
\end{theorem}
\begin{proof}
Without loss of generality, we can assume $C>0$.  And let $s>0$ be the largest time such that
$$C\leq \sup_{M\times [0,T]} u\leq 2C.$$
We define $w$ by
$$w=u-\epsilon(g+C_1 t),$$
where $C_1\geq B$ is a fixed constant, and $\epsilon>0$. We choose $A$ such that
$$\epsilon A>C.$$
Then we have
\begin{equation*}
\begin{split}
\frac{\partial}{\partial t}w-\sl w&=\frac{\partial}{\partial t}u-\sl u-\epsilon(\frac{\partial}{\partial t}g-\sl g+C_1)\\
&\leq \frac{\partial}{\partial t}u-\sl u+\epsilon(B-C_1)\leq 0.
\end{split}
\end{equation*}
Since $w\leq 2C-\epsilon A<C$ for $d_{\th}(x_0,x)>d_0$. Hence, $w$ can attain its supremum at a point $x_1\in M$. Then at point $(x_1,t)$, we have
 $$\sl w\leq 0.$$
From which we get
$$\frac{\partial}{\partial t}\Big(\sup_{M}w(t)\Big)\leq 0.$$
Therefore,
\begin{equation}\label{2.2}
u-\epsilon(g+C_1 t)\leq \sup_{M}w(t)\leq \sup_{M}w(0)\leq C.
\end{equation}
Let $\epsilon$ tends to $0$ in \eqref{2.2}, we obtain for any $t\leq s$,
$$u(x,t)\leq C.$$
Then by iteration, we get the conclusion that $u(x,t)\leq C$ on $M\times [0,T]$.
\end{proof}

Next we prove the following uniqueness theorem of the CR Yamabe flow \eqref{CRYF} on complete noncompact CR manifolds.

\begin{theorem}\label{2.6}
Let $(M,\theta)$ be a complete noncompact CR manifold. Suppose that $u$ is the solution of the CR Yamabe flow \eqref{CRYF2}
such that
$(M,\theta(t)=u^{\frac{2}{n}}\theta)$ satisfies \textit{Condition B} on $M\times [0,T]$.
If $u\in C^2(M\times [0,T])$ and $0<c\leq u \leq C$ on  $M\times [0,T]$,
then  $u$ is the unique solution of the CR Yamabe flow \eqref{CRYF2}.
\end{theorem}
\begin{proof}
The CR Yamabe flow \eqref{CRYF2} is given by
\begin{equation}\label{2.3}
u^{\frac{2}{n}} \frac{\partial u}{\partial t}=(n+1)\left(\sl u-\frac{n}{2(n+1)}R_{\th}u\right),
~~u(0)=u_0>0.
\end{equation}
Now we argue by contradiction. Suppose that there exists two positive solutions $u$,$v$ of this initial value problem \eqref{2.3}.
We define
$$w=u-v.$$
Then, in $M\times [0,T]$, we have
\begin{equation}\label{2.4}
\begin{split}
(n+1)\left(\sl w-\frac{n}{2(n+1)}R_{\th} w\right)&=u^{\frac{2}{n}}\frac{\partial u}{\partial t}-v^{\frac{2}{n}}\frac{\partial v}{\partial t}\\
&=u^{\frac{2}{n}}\frac{\partial w}{\partial t}+\left(u^{\frac{2}{n}}-v^{\frac{2}{n}}\right)\frac{\partial v}{\partial t}\\
&=u^{\frac{2}{n}}\frac{\partial w}{\partial t}+
\left(\int_0^1 \frac{d}{dz}(zu+(1-z)v)^{\frac{2}{n}}dz\right)\frac{\partial v}{\partial t}\\
&=u^{\frac{2}{n}}\frac{\partial w}{\partial t}+
\left(\int_0^1 \frac{2}{n}(zu+(1-z)v)^{\frac{2-n}{n}}(u-v)dz\right)\frac{\partial v}{\partial t}\\
&=u^{\frac{2}{n}}\frac{\partial w}{\partial t}+\left(\int_0^1 (zu+(1-z)v)^{\frac{2-n}{n}}dz\right)
\frac{2}{n}w\frac{\partial v}{\partial t}.
\end{split}
\end{equation}
It follows from \eqref{2.4} that
\begin{equation}\label{2.5}
\begin{split}
\frac{\partial w}{\partial t}&=\frac{n+1}{u^{\frac{2}{n}}}\Delta_{\th}w-\frac{n}{2}R_0wu^{-\frac{2}{n}}\\
&-\frac{2}{n}u^{-\frac{2}{n}}\frac{\partial v}{\partial t}w\int_0^1(zu+(1-z)v)^{\frac{2-n}{n}}dz\\
&=\frac{n+1}{u^{\frac{2}{n}}}\Delta_{\th}w-\left(\frac{2}{nu^{\frac{2}{n}}}\frac{\partial v}{\partial t}\int_0^1(zu+(1-z)v)^{\frac{2-n}{n}}dz+\frac{2}{n}\frac{R_{\th}}{u^{\frac{2}{n}}}\right)w\\
&=a\Delta_{\th}w-bw,\hspace{2mm}\mbox{ on} \hspace{2mm} M\times [0,T],\\
w(0)&\equiv 0.
\end{split}
\end{equation}
Here
$$a=\frac{n+1}{u^{\frac{2}{n}}}>0~~\mbox{ and }~~b=\frac{2}{nu^{\frac{2}{n}}}\frac {\partial v}{\partial t}\int_0^1(zu+(1-z)v)^{\frac{2-n}{n}}dz+\frac{2}{n}\frac{R_{\th}}{u^{\frac{2}{n}}}$$
are continuous functions. Since $u$ and $v$ are solutions of the CR Yamabe flow \eqref{CRYF2},
it follows from the assumptions of Theorem \ref{2.6} that $b$
is bounded from the above. Let $\bar{w}=e^{-\lambda t}w$. Then we have
\begin{equation*}
\begin{split}
\frac{\partial }{\partial t}\bar{w}&=e^{-\lambda t}\frac{\partial w}{\partial t}-\lambda e^{-\lambda t} w\\
&=e^{-\lambda t}\left(\frac{\partial w}{\partial t}-\lambda w\right)\\
&=e^{-\lambda t}[a\Delta_{\th_0}w-(b+\lambda)w]\\
&=a\Delta_{\th_0}\bar{w}-(b+\lambda)\bar{w}.
\end{split}
\end{equation*}
Thus, we choose $\lambda>0$ sufficiently large such that $b+\lambda>0$. If $w$ is not identically zero on $M\times [0,T]$, then without loss of generality, we can assume $w>0$ somewhere, which means $\bar{w}>0$ at the same point. Since $\bar{w}(0)\equiv 0$, then by Theorem \ref{2.7}, we have $\bar{w}\leq 0$. This is a contradiction. Thus we obtain $\bar{w}\equiv 0$ on $M\times [0,T]$, which means $w\equiv 0$ on $M\times [0,T]$.
\end{proof}

\section{The CR Yamabe soliton}\label{section4}

By using Proposition \ref{prop1.1}, we can prove the following theorem.
The corresponding theorem for the Yamabe soliton was proved in \cite{mamiquel}.

\begin{theorem}\label{thm1.1}
Suppose $(M,\theta)$ is a $3$-dimensional noncompact CR Yamabe soliton
which is nonshrinking, i.e. $\mu\leq 0$.
Assume that $\liminf_{x\to\infty}R_\theta(x)\geq 0$.
Then the Tanaka-Webster scalar curvature $R_\theta$ of $(M,\theta)$ is nonnegative.
\end{theorem}
\begin{proof}
Suppose on the contrary that $\inf_M R_\theta<0$. Since $\liminf_{x\to\infty}R_\theta(x)\geq 0$ by assumption,
there exists $x_1\in M$ such that
\begin{equation}\label{1.7}
R_\theta(x_1)=\inf_M R_\theta<0.
\end{equation}
 Then we have
\begin{equation}\label{1.4}
\Delta_\theta R_\theta(x_1)\geq 0,~~\nabla_\theta R_\theta(x_1)=0~~\mbox{ and }~~(R_\theta)_0(x_1)=0.
\end{equation}
Combining (\ref{1.5}) and (\ref{1.4}), we obtain
\begin{equation}\label{1.6}
2R_\theta(R_\theta-\mu)\leq 0~~\mbox{ at }x_1.
\end{equation}
On the other hand, it follows from (\ref{1.7}) and the assumption that $\mu\leq 0$ that
\begin{equation*}
2R_\theta(R_\theta-\mu)=2R_\theta^2-2\mu R_\theta>0
~~\mbox{ at }x_1,
\end{equation*}
which contradicts to (\ref{1.6}).
This proves that $\inf_M R_\theta\geq 0$, as required.
\end{proof}

Next, we consider the CR Yamabe soliton on the
Heisenberg group $\mathbb{H}^{n}$
equipped with the standard contact form.
The standard contact form $\theta_0$ on the Heisenberg group $\mathbb{H}^{n}$ is given by
$$\theta_0=dt+\sqrt{-1}\sum_{i=1}^n(z_id\overline{z}_i-\overline{z}_idz_i).$$
Note that both of the Tanaka-Webster scalar curvature and torsion
of $(\mathbb{H}^{n},\theta_0)$ vanish.
Hence, it follows from (\ref{1.2}) that
any CR Yamabe soliton $(\mathbb{H}^{n},\theta_0,f)$ is given by
\begin{equation}\label{1.9}
\frac{1}{2}f_0=\mu~~\mbox{ and }~~
f_{\alpha\alpha}=0~~\mbox{ for all }\alpha=1,2,...,n,
\end{equation}
for some real-valued function $f$ and some constant $\mu$.
The following theorem classifies all the CR Yamabe solition $(\mathbb{H}^{n},\theta_0,f)$.

\begin{theorem}\label{main}
Suppose that $(\mathbb{H}^{n},\theta_0,f)$ is a CR Yamabe soliton. Then
the potential function $f$ must be in the form of
\begin{equation}\label{1.13}
f(z,t)=2\mu t+\sum_{
p_i,q_i=0,1}C_{p_1q_1\cdots p_nq_n}z_1^{p_1}\overline{z}_1^{q_1}z_2^{p_2}\overline{z}_2^{q_2}\cdots z_n^{p_n}\overline{z}_n^{q_n}
\end{equation}
where $C_{p_1q_1\cdots p_nq_n}$ are constants such that
$f$ is a real-valued function.
\end{theorem}

For example,
we can take
$$C_{p_1q_1\cdots p_nq_n}
=\left\{
   \begin{array}{ll}
     C, & \hbox{when $p_i=q_i=1$ and $p_j=q_j=0$ for $j\neq i$, where $1\leq i\leq n$;} \\
     0, & \hbox{otherwise.}
   \end{array}
 \right.
$$
for some real constant $C$
to get
$f(z,t)=2\mu t+C|z|^2$ satisfying (\ref{1.9}), which has already been pointed out in \cite{Cao}.

\begin{proof}[Proof of Theorem \ref{main}]
To prove Theorem \ref{main}, we are going to show that
any real function $f$ satisfying
(\ref{1.9}) must be in the form of (\ref{1.13}).
We integrate the first equation in (\ref{1.9}) to get
\begin{equation}\label{1.14}
f(z,t)=2\mu t+g
\end{equation}
for some real-valued function $g$ depending on $z$ and $\overline{z}$.
Combining this with the second equation in (\ref{1.9}), we get
$g_{\alpha\alpha}=g_{\overline{\alpha}\,\overline{\alpha}}=0$
for all $\alpha=1,2,...,n$.
This implies
\begin{equation}\label{1.20a}
g_\alpha\mbox{ does not depend on }z_\alpha,
\mbox{ or equivalently, }
g_{\overline{\alpha}}\mbox{ does not depend on }\overline{z}_\alpha.
\end{equation}
Note that (see (2.14) in \cite{Lee})
\begin{equation}\label{1.15}
g_{\alpha\beta}=g_{\beta\alpha}~~\mbox{ and }~~g_{\alpha\overline{\beta}}-g_{\overline{\beta}\alpha}=\sqrt{-1} h_{\alpha\overline{\beta}} g_0
,
\end{equation}
where $h_{\alpha\overline{\beta}}$ is given by
$$d\theta=\sqrt{-1}h_{\alpha\overline{\beta}}\theta_\alpha\wedge\theta_{\overline{\beta}}.$$
Since $g$ does not depend on $t$, we have $g_0=0$.
Hence, it follows from  (\ref{1.15}) that
\begin{equation}\label{1.16}
g_{\alpha\beta}=g_{\beta\alpha}~~\mbox{ and }~~
g_{\alpha\overline{\beta}}=g_{\overline{\beta}\alpha}~~\mbox{ for all }\alpha, \beta.
\end{equation}
It follows from (\ref{1.20a})
that $g_\beta$ does not depend on $z_\beta$,
which implies that
$g_{\beta\alpha}$ does not depend on $z_\beta$ for all $\alpha$.
Since $g_{\alpha\beta}=g_{\beta\alpha}$ by (\ref{1.16}), we can conclude that
$g_{\alpha\beta}$ does not depend on $z_\beta$ for all $\alpha$.
Integrating $g_{\alpha\beta}$ with respect to $z_\beta$, we have,  for any $\alpha=1, 2,..., n$,
\begin{equation}\label{1.17}
g_\alpha= E^\beta(z) z_\beta+ F^\beta(z)
\end{equation}
for some functions $E^\beta$ and $F^\beta$ which do not depend on $z_\beta$.
On the other hand, by (\ref{1.20a}), $g_{\overline{\beta}}$ does not depend on $\overline{z}_\beta$,
which implies that
$g_{\overline{\beta}\alpha}$ does not depend on $\overline{z}_\beta$.
Since $g_{\alpha\overline{\beta}}=g_{\overline{\beta}\alpha}$   by (\ref{1.16}), we can conclude that,
$g_{\alpha\overline{\beta}}$ does not depend on $\overline{z}_\beta$ for all $\alpha$.
Integrating $g_{\alpha\overline{\beta}}$ with respect to $z_\beta$, we have,  for any $\alpha=1, 2,..., n$,
\begin{equation}\label{1.18}
g_\alpha= H^\beta(z)\overline{z}_\beta+I^\beta(z).
\end{equation}
for some functions $H^\beta$ and $I^\beta$ which do not depend on $\overline{z}_\beta$.
Differentiating (\ref{1.17}) with respect to $z_\alpha$ and using
the fact that $g_{\alpha\alpha}=0$, we obtain
\begin{equation}\label{1.19}
0= \frac{\partial E^\beta(z)}{\partial z_\alpha} z_\beta+ \frac{\partial F^\beta(z)}{\partial z_\alpha}
\end{equation}
for $\beta\neq\alpha$.
It follows from (\ref{1.19}) that
$\displaystyle\frac{\partial E^\beta(z)}{\partial z_\alpha}=\frac{\partial F^\beta(z)}{\partial z_\alpha}=0$
for $\beta\neq\alpha$.
In particular, for any $\beta\neq\alpha$,
$E^\beta(z)$ and $F^\beta(z)$ does not depend on $z_\alpha$. Hence, we can integrate
(\ref{1.17}) with respect to $z_\alpha$ to obtain
\begin{equation}\label{1.20}
g(z)= E^\beta(z) z_\beta z_\alpha+ F^\beta(z)z_\alpha+G^\beta(z)~~\mbox{ when }\beta\neq\alpha,
\end{equation}
for some function $G^\beta$ which does not depend on $z_\alpha$
and some functions $E^\beta$ and $F^\beta$ which do not depend on $z_\beta$ and $z_\alpha$.
Similarly, differentiating (\ref{1.18}) with respect to $z_\alpha$ and using
the fact that $g_{\alpha\alpha}=0$, we obtain
\begin{equation}\label{1.21}
0= \frac{\partial H^\beta(z)}{\partial z_\alpha} \overline{z}_\beta+ \frac{\partial I^\beta(z)}{\partial z_\alpha}
\end{equation}
for all $\beta$.
It follows from (\ref{1.21}) that
$\displaystyle\frac{\partial H^\beta(z)}{\partial z_\alpha}=\frac{\partial I^\beta(z)}{\partial z_\alpha}=0$
for all $\beta$.
In particular, for any $\beta$,
$H^\beta(z)$ and $I^\beta(z)$ does not depend on $z_\alpha$. Hence,
we can integrate
(\ref{1.18}) with respect to $z_\alpha$ to obtain
\begin{equation}\label{1.22}
g(z)= H^\beta(z) \overline{z}_\beta z_\alpha+ I^\beta(z)z_\alpha+J^\beta(z)
\end{equation}
for some function $J^\beta$ which does not depend on $z_\alpha$
and some functions $H^\beta$ and $I^\beta$ which do not depend on $\overline{z}_\beta$ and $z_\alpha$.
Now (\ref{1.13}) follows from (\ref{1.20}) and (\ref{1.22}).
\end{proof}

We have the following transformation law for the Tanaka-Webster scalar curvature
and the torsion: if $\hat{\theta}=e^{2u}\theta$, then we have
(see Lemma 5.6 and Proposition 5.15 in \cite{Lee})
\begin{equation}\label{1.3}
\begin{split}
e^{2u}\widehat{R}&=R+2(n+1)\Delta_{\theta} u-4n(n+1)u_\alpha u^\alpha,\\
\widehat{A}_{\alpha\beta}&=e^{-2u}(A_{\alpha\beta}+2iu_{\alpha\beta}-4iu_\alpha u_\beta).
\end{split}
\end{equation}
It follows from
(\ref{1.3}) that
the Tanaka-Webster scalar curvature and torsion of $(\mathbb{H}^{n},\hat{\theta}=e^{2u}\theta_0)$ are given by
\begin{equation}\label{1.8}
\begin{split}
\widehat{R}&=e^{-2u}(2(n+1)\Delta_{\theta_0} u-4n(n+1)u_\alpha u^\alpha),\\
\widehat{A}_{\alpha\beta}&=e^{-2u}(2iu_{\alpha\beta}-4iu_\alpha u_\beta).
\end{split}
\end{equation}
Also, the vector field $\hat{T}$, the frame and the connection of $\hat{\theta}=e^{2u}\theta_0$ are given by (see (5.7) and (5.14) in \cite{Lee})
\begin{equation}\label{1.8a}
\begin{split}
\hat{Z}_{\alpha}&=e^{-u}Z_\alpha,~~\hat{T}=e^{-2u}(T+2iu^{\overline{\gamma}}Z_{\overline{\gamma}}-2iu^\gamma Z_\gamma)\mbox{ and }\\
\widehat{{\omega_\beta}}^\alpha
&={\omega_\beta}^\alpha+2(u_\beta\theta^\alpha-u^\alpha\theta_\beta)
+\delta_\beta^\alpha(u_\gamma\theta^\gamma-u^\gamma\theta_\gamma)\\
&\hspace{4mm}+i({u^\alpha}_\beta+{u_\beta}^\alpha+4f_\beta f^\alpha+4\delta_\beta^\alpha u_\gamma u^\gamma)\theta.
\end{split}
\end{equation}
It follows from (\ref{1.8a}) that
\begin{equation}\label{1.8b}
\begin{split}
\hat{f}_{\alpha\alpha}
&=\hat{Z}_{\alpha}\hat{Z}_{\alpha}f-\widehat{{\omega_\beta}}^\alpha(\hat{Z}_{\alpha})f\\
&=e^{-u}Z_\alpha (e^{-u}Z_\alpha f)-{\omega_\alpha}^\alpha(e^{-u}Z_\alpha) (e^{-u}Z_\alpha f)
-3u_\alpha \theta^\alpha(e^{-u}Z_\alpha) (e^{-u}Z_\alpha f)\\
&=e^{-2u}f_{\alpha\alpha}+e^{-u} Z_\alpha(e^{-u}) Z_\alpha f-3e^{-2u}u_\alpha f_\alpha\\
&=e^{-2u}(f_{\alpha\alpha}-4u_\alpha f_\alpha).
\end{split}
\end{equation}
Now suppose that $(\mathbb{H}^{n},\hat{\theta}=e^{2u}\theta_0,f)$ is a CR Yamabe soliton.
Then, by (\ref{1.2}), (\ref{1.8})-(\ref{1.8b}), we have
\begin{equation}\label{1.10}
\begin{split}
&e^{-2u}(2(n+1)\Delta_{\theta_0} u-4n(n+1)u_\alpha u^\alpha)+\frac{1}{2}e^{-2u}(f_0+
2iu^{\overline{\gamma}}f_{\overline{\gamma}}-2iu^\gamma f_\gamma)=\mu,\\
&e^{-2u}(f_{\alpha\alpha}-4u_\alpha f_\alpha)-e^{-2u}(2u_{\alpha\alpha}-4u_\alpha^2) f=0~~\mbox{ for all }\alpha=1,2,...,n.
\end{split}
\end{equation}

If $u$ depends only on $t$, then (\ref{1.10}) reduces to
\begin{equation}\label{1.30}
\frac{1}{2}e^{-2u}f_0=\mu~~\mbox{ and }~~f_{\alpha\alpha}=0~~\mbox{ for all }\alpha=1,2,...,n.
\end{equation}
Integrating the first equation in (\ref{1.30}) with respect to $t$, we obtain
\begin{equation*}
f(z,t)=2\mu\int e^{2u}dt +g
\end{equation*}
for some real-valued function $g$ depending only on $z$ and $\overline{z}$.
Submitting this into the second equation of (\ref{1.30}), we get
$g_{\alpha\alpha}=g_{\overline{\alpha}\,\overline{\alpha}}=0$
for all $\alpha=1,2,...,n$.
Now, we can follow the proof of Theorem \ref{main} to conclude
the following:

\begin{cor}
Suppose that $(\mathbb{H}^{n},e^{2u}\theta_0,f)$ is a CR Yamabe soliton
such that $u$ is a function depending only on $t$. Then
the potential function $f$ must be in the form of
\begin{equation*}
f(z,t)=2\mu \int e^{2u}dt+\sum_{
p_i,q_i=0,1}C_{p_1q_1\cdots p_nq_n}z_1^{p_1}\overline{z}_1^{q_1}z_2^{p_2}\overline{z}_2^{q_2}\cdots z_n^{p_n}\overline{z}_n^{q_n}
\end{equation*}
where $C_{p_1q_1\cdots p_nq_n}$ are constants such that
$f$ is a real-valued function.
\end{cor}

\bibliographystyle{amsplain}

\end{document}